\documentclass[reqno]{amsart}

\usepackage{amssymb,latexsym,mathrsfs,amsmath}
\usepackage{amsthm}
\usepackage{graphicx}
\usepackage{diagbox}
\usepackage{mathrsfs}
\usepackage{hyperref}
\usepackage{tikz}
\usetikzlibrary{matrix}
\usetikzlibrary{arrows}

\newtheorem{lemma}{Lemma}[section]
\newtheorem{proposition}[lemma]{Proposition}
\newtheorem{theorem}[lemma]{Theorem}

\theoremstyle{definition}

\newtheorem{definition}[lemma]{Definition}

\DeclareMathOperator{\id}{id}
\DeclareMathOperator{\chr}{char}

\DeclareMathOperator{\im}{im}

\newcommand{\Z}{\mathbb{Z}}

\newcommand{\R}{\mathbb{R}}

\newcommand{\F}{\mathbb{F}}

\newcommand{\FBN}{\mathcal{F}_{\mathrm{BN}}}
\newcommand{\CBN}{C_{\mathrm{BN}}}
\newcommand{\HBN}{H_{\mathrm{BN}}}
\newcommand{\Cob}{\mathcal{C}ob_\bullet}
\newcommand{\Cobl}{\mathcal{C}ob_{\bullet/l}}
\newcommand{\Fml}{\mathrm{Foam}_{\bullet/l}}

\newcommand{\plane}[2]{
\shade[color = gray!40, opacity = 0.3] (#1,#2-0.5) -- (#1+1,#2) -- (#1+1,#2+1.5) -- (#1,#2+1) -- (#1,#2-0.5);
\draw (#1,#2-0.5) -- (#1+1,#2) -- (#1+1,#2+1.5) -- (#1,#2+1) -- (#1,#2-0.5);
}

\newcommand{\parallelbox}[4] {
\shade[color = gray!40, opacity = 0.3] (#1,#2) -- (#1 + #3, #2) -- (#1 + #3 + 0.3 * #4, #2 + #4) -- (#1 + 0.3 * #4, #2 + #4) -- (#1, #2);
\draw[-] (#1,#2) -- (#1 + #3, #2) -- (#1 + #3 + 0.3 * #4, #2 + #4) -- (#1 + 0.3 * #4, #2 + #4) -- (#1, #2);
}

\newcommand{\sphere}[3]{
\shade[ball color = gray!40, opacity = 0.3] ({#1},{#2}) circle ({#3});
\draw (#1,#2) circle ({#3});
\draw (#1-#3,#2) arc (180:360:#3 and 0.3*#3);
\draw[dashed] (#1+#3,#2) arc (0:180:#3 and 0.3*#3);
}

\newcommand{\spheredot}[3]{
\sphere{#1}{#2}{#3}
\node at (#1-0.3*#3,#2+0.45*#3) [scale = 1.5*#3] {$\bullet$};
}

\newcommand{\bowl}[4]{
\shade[ball color = gray!40, opacity = 0.3] (#1+0.8*#3,#2) arc (0:180:0.4*#3 and -0.2*#3) -- (#1,#2-#4*#3) arc (180:360:0.4*#3 and 0.4*#3) -- (#1+0.8*#3,#2);
\draw (#1+0.8*#3,#2) arc (0:180:0.4*#3 and -0.2*#3) -- (#1,#2-#4*#3) arc (180:360:0.4*#3 and 0.4*#3) -- (#1+0.8*#3,#2);
\draw (#1,#2) arc (180:360:0.4*#3 and -0.2*#3);
\shade[ball color = gray!40, opacity = 0.15] (#1,#2) arc (180:360:0.4*#3 and -0.2*#3) arc (0:180:0.4*#3 and -0.2*#3);
}

\newcommand{\bowlud}[4]{
\shade[ball color = gray!40, opacity = 0.3] (#1,#2) arc (180:360:0.4*#3 and 0.2*#3) -- (#1+0.8*#3,#2+#4*#3) arc (0:180:0.4*#3 and 0.4*#3) -- (#1,#2);
\draw (#1,#2) arc (180:360:0.4*#3 and 0.2*#3) -- (#1+0.8*#3,#2+#4*#3) arc (0:180:0.4*#3 and 0.4*#3) -- (#1,#2);
\draw[dashed] (#1+0.8*#3,#2) arc (0:180:0.4*#3 and 0.2*#3);
}

\newcommand{\cylinder}[4]{
\shade[ball color = gray!40, opacity = 0.3] (#1,#2) arc (180:360:0.4*#3 and 0.2*#3) -- (#1+0.8*#3,#2+#4*#3) 
arc (0:180:0.4*#3 and -0.2*#3);
\shade[ball color = gray!40, opacity = 0.15] (#1,#2+#4*#3) arc (180:360:0.4*#3 and -0.2*#3) arc (0:180:0.4*#3 and -0.2*#3);
\draw (#1,#2) arc (180:360:0.4*#3 and 0.2*#3) -- (#1+0.8*#3,#2+#4*#3) 
arc (0:180:0.4*#3 and 0.2*#3) -- (#1,#2);
\draw[dashed] (#1+0.8*#3,#2) arc (0:180:0.4*#3 and 0.2*#3);
\draw (#1,#2+#4*#3) arc (180:360:0.4*#3 and 0.2*#3);
}

\newcommand{\pcrossing}[3] {
\draw[->, thick, >=latex] (#1, #2) -- (#1+#3, #2+#3);
\draw[-, thick, >=latex] (#1+#3, #2) -- (#1+ 0.6 * #3, #2 + 0.4 * #3);
\draw[->, thick, >=latex] (#1+ 0.4 * #3, #2 + 0.6 * #3) -- (#1, #2+#3);
}

\newcommand{\mcrossing}[3] {
\draw[->, thick, >=latex] (#1+#3, #2) -- (#1, #2+#3);
\draw[-, thick, >=latex] (#1, #2) -- (#1+ 0.4 * #3, #2 + 0.4 * #3);
\draw[->, thick, >=latex] (#1+ 0.6 * #3, #2 + 0.6 * #3) -- (#1+#3, #2+#3);
}

\newcommand{\nssmoothing}[4] {
\draw[#4, thick, >=latex] (#1, #2) to [out = 45, in = 270] (#1+0.3 * #3, #2+0.5 * #3) to [out = 90, in = 315] (#1, #2+#3);
\draw[#4, thick, >=latex] (#1+#3, #2) to [out = 135, in = 270] (#1+0.7* #3, #2+0.5 * #3) to [out = 90, in = 225] (#1+#3, #2+#3);
}

\newcommand{\wesmoothing}[4] {
\draw[#4, thick, >=latex] (#1, #2) to [out = 45, in = 180] (#1+0.5 * #3, #2+0.3 * #3) to [out = 0, in = 135] (#1+#3, #2);
\draw[#4, thick, >=latex] (#1, #2+#3) to [out = 315, in = 180] (#1+0.5 * #3, #2+0.7 * #3) to [out = 0, in = 225] (#1+#3, #2+#3);
} 

\newcommand{\thetafoam}[3] {
\fill[color = gray, opacity = 0.7] (#1-#3,#2) arc (180:360:#3 and 0.3*#3) arc (0:180:#3 and 0.3*#3);
\shade[ball color = gray!40, opacity = 0.3] ({#1},{#2}) circle ({#3});
\draw (#1,#2) circle ({#3});
\draw[thick] (#1-#3,#2) arc (180:360:#3 and 0.3*#3);
\draw[dashed, thick] (#1+#3,#2) arc (0:180:#3 and 0.3*#3);
\draw[<-, >=latex, thick] (#1-0.2*#3,#2-0.3 * #3) -- (#1, #2-0.3 * #3);
}

\newcommand{\spidering}[4] {
\draw[#4, thick, >=latex] (#1, #2) -- (#1+0.5*#3, #2+0.25*#3);
\draw[#4, thick, >=latex] (#1+#3, #2) -- (#1+0.5*#3, #2+0.25*#3);
\draw[#4, thick, >=latex] (#1+0.5*#3, #2+0.75*#3) -- (#1, #2+#3);
\draw[#4, thick, >=latex] (#1+0.5*#3, #2+0.75*#3) -- (#1+#3, #2+#3);
\draw[#4, thick, >=latex] (#1+0.5*#3, #2+0.75*#3) -- (#1+0.5*#3, #2+0.25*#3);
}

\newcommand{\smoothing}[4] {
\draw[#4, thick, >=latex] (#1, #2) to [out = 45, in = 315] (#1, #2+#3);
\draw[#4, thick, >=latex] (#1+#3, #2) to [out = 135, in = 225] (#1+#3, #2+#3);
}

\begin{document}
\parindent0em
\setlength\parskip{.1cm}
\thispagestyle{empty}
\title{Efficient calculations of $S$-invariants for links}
\author[Dirk Sch\"utz]{Dirk Sch\"utz}
\address{Department of Mathematical Sciences\\ Durham University\\ Durham DH1 3LE\\ United Kingdom}
\email{dirk.schuetz@durham.ac.uk}
%\subjclass[2020]{primary: 57K18; secondary: 57K10}
%\keywords{Odd Khovanov homology, Steenrod square}

\begin {abstract}
We describe an algorithm that can effectively calculate the $s$-invariant of a link as defined by Beliakova and Wehrli. Our computations show that this cannot be done by merely calculating the $E_\infty$-page of the Bar-Natan--Lee--Turner spectral sequence. Our methods also work for $s$-invariants coming from $\mathfrak{sl}_3$-link homology.
\end{abstract}

\maketitle

\section{Introduction}
The $s$-invariant for knots was introduced by Rasmussen \cite{MR2729272} and has led to major applications of Khovanov homology \cite{MR1740682} to $4$-dimensional topology via a lower bound of the slice genus. A generalization to links was given by Beliakova and Wehrli \cite{MR2462446} which is a weak-slice obstruction for links.

For knots, the $s$-invariant can be calculated using the Bar-Natan--Lee--Turner spectral sequence whose $E_1$-page is Khovanov homology and whose $E_\infty$-page resembles the Khovanov homology of the unknot. This spectral sequence also exists for links, but the $E_\infty$-page is a bit more complicated. It was shown by Pardon \cite{MR2928905} that the $E_\infty$-page is a link-concordance invariant. This spectral sequence is amenable to fast calculations via Bar-Natan's approach \cite{MR2320156}, but in the case of links it may only give upper and lower bounds for the $s$-invariant.

As a result, computations of $s$-invariants for links are harder to come by. A striking example was given by Ren \cite{MR4843752}, who computed the $s$-invariants of torus links with any orientation and used it to establish an adjunction type inequality conjectured in \cite{MR4541332}. The main purpose of this paper is to give a way to compute the $s$-invariant of a link, $s^p(L)$, where $p$ is the characteristic of a field, which is comparable in complexity to calculating the $s$-invariant of a knot.

The aforementioned lower bound for $s^p(L)$ given through the $E_\infty$-page of the spectral sequence is given by
\[
s^p_{\min}(L) = \min \{j\in \Z\mid E_\infty^{0,j}\not=0\}+1,
\]
and the upper bound is given by
\[
s^p_{\max}(L) = \max\{j\in \Z\mid E_\infty^{0,j}\not=0\}-1.
\]
For a knot, $s^p_{\min}(L) = s^p_{\max}(L)$, and this can also happen for links. But whenever $s^p_{\max}(L) > s^p_{\min}(L)$, the lower bound is usually the better approximation for $s^p(L)$. Indeed, we have the following.

\begin{proposition}\label{prp:smin}
Let $L$ be an oriented link $L$ and $p$ a prime or $0$. Then there is an orientation on $L$ such that for the resulting oriented link $\tilde{L}$ we have $s^p(\tilde{L}) = s^p_{\min}(L)$. Moreover, the total linking number of $\tilde{L}$ and $L$ are the same.
\end{proposition}

In many cases we have $\tilde{L}=L$, for example if $s^p_{\max}(L) = s^p_{\min}(L)$, or if $L$ is a torus link with arbitrary orientation \cite[Thm.1.2]{MR4843752} (in which case often $s^p_{\max}(L) > s^p_{\min}(L)$). Nevertheless, there are exceptions.

\begin{theorem}\label{th:nonconst}
There exists a $2$-component link $L$ with linking number $0$ such that $s^p(L)$ depends on the orientations on $L$.
\end{theorem}

We note that this is not the case for the signature by \cite[Thm.1]{MR261585}.

Our algorithm carries over to $s$-invariants defined for $\mathfrak{sl}_3$-link homology \cite{MR2100691}. Both algorithms have been implemented in a computer programme which is available from the author's website.

\section{Frobenius systems and Bar-Natan homology}\label{sec:Frob}
Let $\F$ be a field. By a {\em Frobenius system} we mean a tuple $\mathcal{F} = (\F, A, \varepsilon, \Delta)$ with $A$ an $\F$-algebra with $1$, $\varepsilon\colon A\to \F$ a linear map, and $\Delta\colon A\to A\otimes_\F A$ an $A$-bimodule map that is co-associative and co-commutative, such that $(\varepsilon\otimes \id_A)\circ \Delta = \id_A$.

As an example, let $\FBN = (\F, \F[X]/(X^2-X), \varepsilon, \Delta)$, where
\[
\Delta(1) = X\otimes 1+1\otimes X - 1\otimes 1,
\]
and
\[
\varepsilon(1) = 0, \hspace{0.5cm}\varepsilon(X) = 1.
\]

Another example is given by $\mathcal{F}_L = (\F, \F[X]/(X^2-1), \varepsilon, \bar{\Delta})$, where
\[
\bar{\Delta}(1) = X\otimes 1+1\otimes X,
\]
and $\varepsilon$ is as before.

If $A$ is $2$-dimensional, Khovanov \cite{MR2232858} showed that an oriented link diagram $D$ gives rise to a cochain complex $C(D;\mathcal{F})$ over $\F$ whose homology is an oriented link invariant. In the case of $\FBN$ we denote the corresponding complex by $\CBN(D;\F)$ and its homology by $\HBN(L;\F)$, where $L$ is the link represented by $D$. We call this complex the {\em Bar-Natan complex} and the homology the {\em Bar-Natan homology}.

Bar-Natan \cite{MR2174270} has given a tangle version of which we are going to describe a simplified version that is sufficient to perform fast calculations of link homology groups.

Let $B\subset \R^2$ be a closed disc and $\dot{B}\subset \partial B$ an oriented $0$-dimensional compact manifold bordant to the empty set.

Let $\Cob(B, \dot{B})$ be the category whose objects are compact smooth $1$-dimensional submanifolds $S\subset B$ with $\partial S = \dot{B}$ and $S$ intersects $\partial B$ transversely. Morphisms between objects $S_0$ and $S_1$ are isotopy classes of dotted cobordisms between $S_0\times\{0\}$ and $S_1\times\{1\}$ embedded in $B\times [0,1]$ and which are a product cobordism near $\dot{B}\times [0,1]$. Here dotted means that there are finitely many specified points in the interior of the cobordism, which can move freely there.

We denote by $\Cobl^\F(B,\dot{B})$ the additive category whose objects are finite dimensional based $\F$-vector spaces where basis elements are objects $q^jS$ with $S$ an object of $\Cob(B,\dot{B})$ and $j\in \Z$. Here $q^j$ stands for a shift in the $q$-grading. Morphisms are given by matrices $(M_{nm})$ with each matrix entry $M_{nm}$ an element of the $\F$-vector space generated by the morphism set between $S_n$ and $S_m$ in $\Cob(B,\dot{B})$, modulo the following local relations:
\begin{equation}\label{eq:localrels}
\begin{tikzpicture}[baseline={([yshift=-.5ex]current bounding box.center)}]
\sphere{0}{0}{0.5}
\node at (1.1,0) {$= 0$,};
\spheredot{3}{0}{0.5}
\node at (4.1,0) {$= 1$,};
\plane{5}{-0.5}
\plane{7.1}{-0.5}
\node at (5.5,0.2) [scale = 0.75] {$\bullet$};
\node at (5.5,-0.2) [scale = 0.75] {$\bullet$};
\node at (7.6,0) [scale = 0.75] {$\bullet$};
\node at (6.6,0) {$=$};
\end{tikzpicture}
\end{equation}
and
\begin{equation}\label{eq:localrelb}
\begin{tikzpicture}[baseline={([yshift=-.5ex]current bounding box.center)}]
\cylinder{0}{0}{1}{1.5}
\node at (1.25,0.75) {$=$};
\bowl{1.6}{1.5}{1}{0.1}
\bowlud{1.6}{0}{1}{0.1}
\node at (2,0.1) [scale = 0.75] {$\bullet$};
\node at (2.85,0.75) {$+$};
\bowl{3.2}{1.5}{1}{0.1}
\bowlud{3.2}{0}{1}{0.1}
\node at (3.6,1.2) [scale = 0.75] {$\bullet$};
\node at (4.45,0.75) {$-$};
\bowl{4.8}{1.5}{1}{0.1}
\bowlud{4.8}{0}{1}{0.1}
\end{tikzpicture}
\end{equation}

If we think of a dot as multiplication by $X$, these relations resemble the Frobenius system $\FBN$.

Given a tangle $T$ consisting of an oriented crossing in a disc $B$ with $4$ points $\dot{B}$, we can obtain a cochain complex $\CBN(T;\F)$ as in Figure \ref{fig:crossing_complex}.
\begin{figure}[ht]
\[
\begin{tikzpicture}
\pcrossing{0}{2}{1}
\mcrossing{0}{0}{1}
\node at (1.5, 2.5) {$:$};
\node at (2.5, 2.5) {$u^0q^1$};
\nssmoothing{3}{2}{1}{ }
\draw[->] (4.2, 2.5) -- node [above] {$S$} (5.3, 2.5);
\node at (6, 2.5) {$u^1q^2$};
\wesmoothing{6.5}{2}{1}{ }
\node at (1.5, 0.5) {$\colon$};
\node at (2.5, 0.5) {$u^{-1}q^{-2}$};
\wesmoothing{3}{0}{1}{ }
\draw[->] (4.2, 0.5) -- node [above] {$S$} (5.3, 0.5);
\node at (6, 0.5) {$u^0q^{-1}$};
\nssmoothing{6.5}{0}{1}{ }
\end{tikzpicture}
\]
\caption{\label{fig:crossing_complex}The complexes associated to a positive and a negative crossing. Here $S$ stands for the standard saddle cobordism. The letters $u$ and $q$ indicate the homological and $q$-grading of the generators.}
\end{figure}

The main property from \cite{MR2174270} that we need is that these tangle complexes can be combined by a 'tensor product'. We will only describe a very basic version here. Given an oriented link diagram $D$, we can always find a sequence of tangle diagrams
\begin{equation}\label{eq:nice_tangles}
\emptyset = T_0 \subset T_1 \subset \cdots \subset T_n = D,
\end{equation}
with $T_i - T_{i-1}$ being the diagram of a single crossing. Furthermore, we can assume that each $T_i \subset B_i$ with $B_i$ a closed disc, $B_i = B_{i-1} \cup \tilde{B}_i$, $B_{i-1}\cap \tilde{B}_i = J_i$ an interval, and $\tilde{B}_i$ a closed disc containing $T_i-T_{i-1}$.

Given a cochain complex $C_{i-1}$ over $\Cobl^\F(B_{i-1},\dot{B}_{i-1})$ and $\tilde{C}_i$ over $\Cobl^\F(\tilde{B}_i,\dot{\tilde{B}}_i)$ we can form the complex $C_{i-1}\otimes \tilde{C}_i$ over $\Cobl^\F(B_i, \dot{B}_i)$, and this construction behaves well with respect to chain homotopy equivalences \cite[Thm.2]{MR2174270}. 

In particular, if we start with $C_1=\CBN(T_1;\F)$ and apply this construction inductively to $\tilde{C}_i=\CBN(T_i;\F)$, the final complex $C_n$ is a complex over $\Cobl^{\F}(B,\emptyset)$, where $B$ is a disc containing the link diagram $D$. Moreover, by the delooping result of Naot \cite[Prop.5.1]{MR2263052} and the local relations we can identify $\Cobl^{\F}(B,\emptyset)$ with the category of finite dimensional graded $\F$-vector spaces, $\mathfrak{V}_\F^q$. Denote this identification functor by $G$. On objects it satisfies $G(q^j S) = q^j A^{\otimes |S|}$, where $|S|$ is the number of components of the closed $1$-manifold $S$, and $A=\F[X]/(X^2-X)$ the graded $\F$-vector space whose $q$-grading is determined by $|1|_q = -1$ and $|X|_q=1$. In particular, we have $G(C_n) = \CBN(D;\F)$. At this stage it is worth pointing out that we do not require linear maps in $\mathfrak{V}^q_\F$ to be grading preserving. In particular, $\CBN(D;\F)$ has bigraded cochain groups, but is not a bigraded cochain complex.

Since the Frobenius algebra $\F[X]/(X^2-X)$ is diagonalizable in the sense of \cite{MR4079621}, the homology $\HBN(L;\F)$ has total dimension $2^{|L|}$, where $|L|$ denotes the number of components, by \cite[Thm.1]{MR4079621}. Furthermore, there is a canonical basis based on the orientations of $L$.

Denote the orientation that comes with our oriented link diagram by $\mathcal{O}$. There exists a unique smoothing $S_\mathcal{O}$ which preserves the orientations of the link diagram, and from Figure \ref{fig:crossing_complex} we see that this smoothing is in homological degree $0$. Then 
\[
G(S_\mathcal{O}) = A^{\otimes |S_\mathcal{O}|}.
\]
The orientation $\mathcal{O}$ gives rise to an element $\mathfrak{s}_\mathcal{O}\in q^{n_+(D)-n_-(D)}G(S_\mathcal{O})\subset \CBN(D;\F)$, where $n_+(D)$ is the number of positive crossings in $D$ and $n_-(D)$ the number of negative crossings in $D$.  To define this element, first put a checkerboard coloring on the discs bounded by the link diagram. This coloring induces a coloring of $S_\mathcal{O}\subset \R^2$ as in Figure \ref{fig:coloring}.
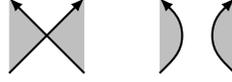
\begin{figure}[ht]
\[
\begin{tikzpicture}
\fill[color = gray!50] (1,0) -- (0.5,0.5) -- (1,1);
\fill[color = gray!50] (0,0) -- (0.5,0.5) -- (0,1);
\draw[->, thick, >=latex] (0,0) -- (1,1);
\draw[->, thick, >=latex] (1,0) -- (0,1);
\fill[color = gray!50] (2,0)  to [out=45, in = 270] (2.3, 0.5) to [out=90, in = 315] (2,1);
\fill[color = gray!50] (3,0) to [out=135, in = 270] (2.7, 0.5) to [out=90, in = 225] (3,1);
\draw[->, thick, >=latex] (2,0) to [out=45, in = 270] (2.3, 0.5) to [out=90, in = 315] (2,1);
\draw[->, thick, >=latex] (3,0) to [out=135, in = 270] (2.7, 0.5) to [out=90, in = 225] (3,1);
\end{tikzpicture}
\]
\caption{\label{fig:coloring}The checkerboard coloring on $D$ and the induced coloring of $S_\mathcal{O}$.}
\end{figure}

Each component $C$ of $S_\mathcal{O}$ has an orientation, and separates a black region from a white region. We define $I_C = X$, if the white region is to the right of the orientation, and $I_C=1-X$, if the white region is to the left of the orientation. Then $\mathfrak{s}_\mathcal{O}$ is defined as
\[
\mathfrak{s}_\mathcal{O} = I_{C_1}\otimes I_{C_2}\otimes \cdots\otimes I_{C_{|S_\mathcal{O}|}} \in q^{n_+(D)-n_-(D)} A^{\otimes|S_\mathcal{O}|},
\]
Note that if $C_i$ and $C_j$ are separated by a crossing, then $I_{C_i} \cdot I_{C_j} = X\cdot (1-X) = 0$, and so $\mathfrak{s}_\mathcal{O}$ represents a cocycle in $\CBN(D;\F)$. This cocycle represents the basis element of $\HBN(L;\F)$ corresponding to the orientation $\mathcal{O}$ \cite{MR4079621}. Since $\mathcal{O}$ is the orientation of $L$, its homological degree is $0$. We also denote this element by $\mathfrak{s}_\mathcal{O}\in \HBN^0(L;\F)$. 

As we already mentioned, the complex $\CBN(D;\F)$ is not graded, but the individual chain groups are. Let us denote the subspace of $\CBN(D;\F)$ consisting of elements of $q$-grading $j$ by $\CBN^{\ast,j}(D;\F)$. Then let
\[
\mathscr{F}_j=\mathscr{F}_j(D;\F) = \bigoplus_{i \geq j} \CBN^{\ast,i}(D;\F)
\]
Notice that $\CBN^{\ast,j}(D;\F)=0$ for $j\in 1+|L|+2\Z$ so we only need to consider $j\in |L|+2\Z$. Then each $\mathscr{F}_j$ is a subcomplex of $\CBN(D;\F)$ with
\[
0 \subset \cdots \subset\mathscr{F}_j \subset \mathscr{F}_{j-2} \subset \cdots \subset \CBN(D;\F).
\]
Furthermore, $\mathscr{F}_j$ is equal to $0$ for large $j$ and equal to $\CBN(D;\F)$ for small $j$. This filtration now induces a filtration on $\HBN(L;\F)$ by
\[
\HBN(L;\F)_j = \im (H(\mathscr{F}_j) \to \HBN(L;\F)) \subset \HBN(L;\F),
\]
and we define the $q$-grading on $\HBN(L;\F)$ by
\[
\HBN(L;\F)^{(j)} = \HBN(L;\F)_j / \HBN(L;\F)_{j+2}.
\]
For a non-zero element $x\in \HBN(L;\F)$ we define the {\em $q$-grading of} $x$ as
\[
|x|_q = \max\{j\in |L|+2\Z\mid x\in \HBN(L;\F)_j\}.
\]

\begin{definition}\label{def:s-inv}
Let $L$ be an oriented link and $\F$ a field of characteristic $p$. The {\em $s$-invariant } of $L$ is defined as
\[
s^p(L) = |\mathfrak{s}_\mathcal{O}|_q+1,
\]
with $\mathfrak{s}_\mathcal{O}\in \HBN^0(L;\F)$ as above.
\end{definition}

That the $s$-invariant only depends on the characteristic of $\F$ follows from the universal coefficient theorem. In particular, we may always choose $\F$ as a prime field.

But that this definition agrees with the definition of Beliakova--Wehrli \cite[\S 7]{MR2462446} is not immediately obvious. Firstly, Beliakova and Wehrli use the Lee complex, which uses the Frobenius system $\mathcal{F}_L$, and secondly they use the average of two $q$-gradings corresponding to elements $\bar{\mathfrak{s}}_\mathcal{O} \pm \bar{\mathfrak{s}}_{\bar{\mathcal{O}}}$ of Lee homology, where $\bar{\mathcal{O}}$ is the orientation obtained by reversing the directions on all components of $L$.

To see that these definitions agree, first note that if $\chr \F \not=2$ the ring isomorphism $\eta\colon \F[X]/(X^2-X) \to \F[X]/(X^2-1)$ sending $X$ to $\frac{1}{2}(X+1)$ induces an isomorphism $\bar{\eta}\colon \CBN(D;\F)\to C(D;\mathcal{F}_L)$ by \cite{MR2232858}, which sends $\mathfrak{s}_\mathcal{O}$ to $\bar{\mathfrak{s}}_\mathcal{O}$. There is also an involution $I\colon \CBN(D;\F)\to\CBN(D;\F)$ with $I(\mathfrak{s}_\mathcal{O}) = \pm \mathfrak{s}_{\bar{\mathcal{O}}}$, see \cite[\S 2]{MR4873797}.

\begin{lemma}\label{lm:average_s}
Let $L$ be a link and $\F$ a field of characteristic $p \not=2$. Then
\[
s^p(L) = \frac{|\mathfrak{s}_\mathcal{O}+\mathfrak{s}_{\bar{\mathcal{O}}}|_q + |\mathfrak{s}_\mathcal{O}-\mathfrak{s}_{\bar{\mathcal{O}}}|_q}{2}.
\]
\end{lemma}

\begin{proof}
Let $j = |\mathfrak{s}_\mathcal{O}|_q$, so that $s^p(L) = j+1$. Since $I$ preserves the filtration on $\CBN(D;\F)$, we have $j = |\mathfrak{s}_{\bar{\mathcal{O}}}|_q$. Furthermore, $\mathfrak{s}_\mathcal{O}+\mathfrak{s}_{\bar{\mathcal{O}}}, \mathfrak{s}_\mathcal{O}-\mathfrak{s}_{\bar{\mathcal{O}}}\in \mathscr{F}_j$. By \cite[Lm.2.1]{MR4873797}, one of $\mathfrak{s}_\mathcal{O}+\mathfrak{s}_{\bar{\mathcal{O}}}$ and $\mathfrak{s}_\mathcal{O}-\mathfrak{s}_{\bar{\mathcal{O}}}$ is in $\mathscr{F}_{j+2}$. Assume that $\mathfrak{s}_\mathcal{O}+\mathfrak{s}_{\bar{\mathcal{O}}}\in \mathscr{F}_{j+2}$.

Now pick a basepoint on $D$ so that $\CBN(D;\F)$ has the structure of a $\F[X]/(X^2-X)$-complex. We can pick the basepoint so that $X\mathfrak{s}_\mathcal{O} = \mathfrak{s}_\mathcal{O}$ and $X\mathfrak{s}_{\bar{\mathcal{O}}}=0$. That is, we put the basepoint on a component $C$ with $I_C = X$. If we had $\mathfrak{s}_\mathcal{O}+\mathfrak{s}_{\bar{\mathcal{O}}} \in \mathscr{F}_i$ with $i>j+2$, then $\mathfrak{s}_\mathcal{O} = X(\mathfrak{s}_\mathcal{O}+\mathfrak{s}_{\bar{\mathcal{O}}})\in \mathscr{F}_{i-2}\subset \mathscr{F}_{j+2}$, a contradiction. Therefore $|\mathfrak{s}_\mathcal{O}+\mathfrak{s}_{\bar{\mathcal{O}}}|_q = j+2$. We also have $|\mathfrak{s}_\mathcal{O}-\mathfrak{s}_{\bar{\mathcal{O}}}|_q = j$, for if $\mathfrak{s}_\mathcal{O}-\mathfrak{s}_{\bar{\mathcal{O}}}\in \mathscr{F}_{j+2}$, we get $2\mathfrak{s}_\mathcal{O}\in \mathscr{F}_{j+2}$. Since $2\in \F$ is invertible by assumption, this would imply $\mathfrak{s}_\mathcal{O}\in \mathscr{F}_{j+2}$. The desired result follows.
\end{proof}

Because of the isomorphism $\eta$ between the Bar-Natan complex and the Lee complex, our definition agrees with the definition in \cite{MR2462446}. Note that Lemma \ref{lm:average_s} does not hold in characteristic $2$.

The basic properties of \cite[\S 7]{MR2462446} carry over to our setting, including the case of characteristic $2$. The next lemma collects these results.

\begin{lemma}\label{lm:basic_prop}
Let $\F$ be a field of characteristic $p$. Then
\begin{enumerate}
\item We have $s^p(O_n) = 1-n$, where $O_n$ is the unlink with $n$ components.
\item We have 
\[
s^p(L_1\sqcup L_2) = s^p(L_1)+s^p(L_2) - 1,
\]
where $\sqcup$ denotes the split union between links $L_1$ and $L_2$.
\item We have
\[
s^p(L_1)+s^p(L_2) - 2 \leq s^p(L_1\# L_2) \leq s^p(L_1) + s^p(L_2),
\]
where $\#$ denotes connected sum between links $L_1$ and $L_2$.
\item We have
\[
-2|L| + 2 \leq s^p(L)+s^p(\overline{L}) \leq 2,
\]
where $\overline{L}$ is the mirror of the link $L$.
\item We have
\[
|s^p(L_2)-s^p(L_1)| \leq -\chi(S),
\]
where $S$ is a smooth oriented cobordism from $L_1$ to $L_2$ such that every connected component of $S$ has boundary in $L_1$.\hfill\qedsymbol
\end{enumerate}
\end{lemma}

The last point leads to a slice obstruction as follows. An oriented link $L$ is {\em slice in the weak sense}, if there exists an oriented smooth connected surface $P\subset B^4$ of genus $0$ such that $\partial P = L$.

\begin{lemma}[{\cite[Lm.7.2]{MR2462446}}]\label{lm:slice_obst}
Let $L$ be slice in the weak sense. Then
\[
\pushQED{\qed}
|s^p(L)| \leq |L|-1.\qedhere
\popQED
\]
\end{lemma}

\begin{proof}[Proof of Proposition \ref{prp:smin}]
Let $\mathcal{O}$ be the orientation on $L$. The elements $\mathfrak{s}_{\mathcal{O}'}$ with varying orientations $\mathcal{O}'$ generate all of the homology $\HBN(L;\F)$, and those orientations with the same total linking number as $\mathcal{O}$ generate $\HBN^0(L;\F)$. In particular, one of them must satisfy $|\mathfrak{s}_{\mathcal{O}'}|_q+1 = s^p_{\min}(L)$.
\end{proof}

\section{Computing $s$-invariants}\label{sec:fast_comp}

To compute the $s$-invariant of a link $L$, note that the filtered complex $\CBN(D;\F)$ gives rise to a spectral sequence $(E_n^{i,j})$ whose $E_1$-page is Khovanov homology of $L$ with $\F$-coefficients and whose $E_\infty$ page recovers $\HBN(L;\F)$. Since $\F$ is a field, the $E_{i+1}$-page is obtained from the $E_i$ page by Gaussian elimination of the boundary parts that change $q$-degree by $2i$. This starts with $E_0 = \CBN(D;\F)$ and each subsequent $E_i$-page can be thought of as a cochain complex chain homotopy equivalent to $\CBN(D;\F)$.

Since the cocycle $\mathfrak{s}_\mathcal{O}$ represents a non-zero generator of $\HBN^0(L;\F)$, we only need to keep track of it during the Gaussian eliminations until we get its representative in $E^{0,\ast}_\infty$. Note that $E^{0,j}_\infty = \HBN^0(L;\F)^{(j)}$, so we can read off $s^p(L)$ from the representative.

In the case of a knot $K$ we do not keep track of $\mathfrak{s}_\mathcal{O}$, since the final page $E_\infty$ only has two non-zero vector spaces $E^{0,s\pm 1}_\infty$ with $s$ being $s^p(K)$. Bar-Natan's algorithm for fast Khovanov homology calculations \cite{MR2320156} can be adapted to obtain fast computations for the $E_\infty$ page, compare for example \cite{MR4244204}. However, for a link $L$ simply knowing $E_\infty$ does not have to determine $s^p(L)$.

In order to keep track of $\mathfrak{s}_\mathcal{O}$, note that we can think of this cocycle as the image of a cochain map $\varphi\colon C \to \CBN(D;\F)$, where $C$ is a cochain complex concentrated in homological degree $0$ and with $C^0 = \F$, and $\varphi(1) = \mathfrak{s}_\mathcal{O}$. This idea was applied to tangles by Sano \cite{sano2025diag} in order to use Bar-Natan's algorithm for fast computations. We now describe a slightly modified version of \cite{sano2025diag}.

If $T_+\subset B$ is the tangle consisting of a positive crossing only, with $\dot{B}$ four points, let $C_{T_+}$ be the cochain complex over $\Cobl^\F(B,\dot{B})$ concentrated in homological degree $0$, generated by one object consisting of the oriented resolution of $T_+$. Similarly, let $C_{T_-}$ be the analogous cochain complex for a negative crossing $T_-\subset B$. We want to describe cochain maps $\varphi_\pm \colon C_{T_\pm} \to \CBN(T_\pm ;\F)$. Pictorially, this looks as in Figure \ref{fig:cochain_map}.
\begin{figure}[ht]
\[
\begin{tikzpicture}
\wesmoothing{0}{0}{1}{-}
\nssmoothing{3}{0}{1}{->}
\nssmoothing{3}{2}{1}{->}
\nssmoothing{6}{0}{1}{->}
\wesmoothing{9}{0}{1}{-}
\nssmoothing{6}{2}{1}{->}
\draw[->] (1.2,0.5) -- node [above] {$S$} (2.8,0.5);
\draw[->] (7.2,0.5) -- node [above] {$S$} (8.8,0.5);
\draw[->] (3.5, 2) -- node [right] {$\varphi_+$} (3.5, 1);
\draw[->] (6.5, 2) -- node [right] {$\varphi_-$} (6.5, 1);
\end{tikzpicture}
\]
\caption{\label{fig:cochain_map}The cochain maps for a single crossing.}
\end{figure}
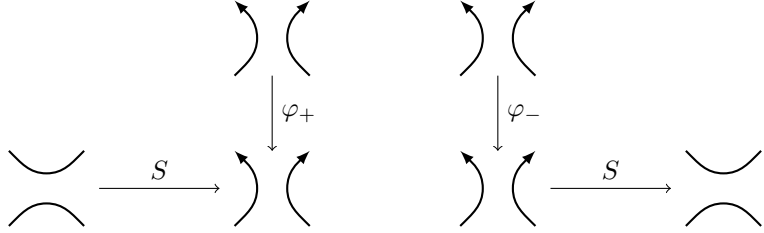

If the tangle is part of an oriented link diagram, the two arcs of the oriented smoothing belong to different components $C_1$ and $C_2$ with $I_{C_1}\not= I_{C_2}$, compare Figure \ref{fig:coloring}. The cochain map $\varphi_\pm$ mimics $I_{C_1}\otimes I_{C_2}$ by
\[
\begin{tikzpicture}
\node at (0,0.4) {$\varphi_\pm = $};
\nssmoothing{0.6}{0}{0.8}{-}
\nssmoothing{1.8}{0}{0.8}{-}
\node at (1.6, 0.4) {$-$};
\node at (1.15,0.4) {$\bullet$};
\node at (2.05,0.4) {$\bullet$};
\node at (2.35,0.4) {$\bullet$};
\end{tikzpicture}
\]
if the left arc is part of a component with $I(C) = 1-X$, and
\[
\begin{tikzpicture}
\node at (0,0.4) {$\varphi_\pm = $};
\nssmoothing{0.6}{0}{0.8}{-}
\nssmoothing{1.8}{0}{0.8}{-}
\node at (1.6, 0.4) {$-$};
\node at (0.85,0.4) {$\bullet$};
\node at (2.05,0.4) {$\bullet$};
\node at (2.35,0.4) {$\bullet$};
\end{tikzpicture}
\]
otherwise. Each summand stands for a product cobordism with dots as indicated.

Now if the oriented link diagram is a sequence of tangles as in (\ref{eq:nice_tangles}), we can form tensor products with the cochain maps $\varphi_\pm$ and obtain a cochain map
\[
\varphi\colon C = C_{T_1}\otimes C_{T_2-T_1}\otimes \cdots \otimes C_{T_n-T_{n-1}} \to \CBN(D;\F),
\]
where $C$ can be identified with the Bar-Natan complex $\CBN(S_\mathcal{O};\F)$. Notice that the orientation smoothing $S_\mathcal{O}$ is an unlink diagram without crossings. This complex is concentrated in homological degree $0$ with $C^0 = (\F[X]/(X^2-X))^{\otimes |S_\mathcal{O}|}$. Furthermore, $\varphi(1\otimes 1\otimes \cdots \otimes 1) = \mathfrak{s}_\mathcal{O}$. This is because $X\cdot X = X$ and $(1-X)^2 = 1-X$ in $\F[X]/(X^2-X)$ and the related dot-relations.

Instead of only doing the tensor products, we can also apply the delooping and Gaussian elimination operations of the Bar-Natan algorithm, which results in a cochain map
\[
\varphi\colon C = C_{T_1}\otimes C_{T_2-T_1}\otimes \cdots \otimes C_{T_n-T_{n-1}} \to D,
\]
where $D$ is chain homotopy equivalent to $\CBN(D;\F)$ as a filtered complex. We can also deloop in the complexes $C_{T_1}\otimes C_{T_2-T_1}\otimes \cdots \otimes C_{T_i-T_{i-1}}$ as we go through the crossings. Since we are only interested in $\varphi(1\otimes \cdots\otimes 1)$, rather than resolving each circle with $\F[X]/(X^2-X)\cong \F^2$, we only need to keep track of the generator $1\in \F[X]/(X^2-X)$. 

So delooping each circle in $C_{T_1}\otimes C_{T_2-T_1}\otimes \cdots \otimes C_{T_i-T_{i-1}}$ to only the $+1$-generator results in a cochain map $\varphi\colon u^0\F\to D$ with $D$ chain homotopy equivalent to $\CBN(D;\F)$ as a filtered complex, and $\varphi(1)$ represents the homology class of $\mathfrak{s}_\mathcal{O}$. Here $u^n\F$ stands for a cochain complex concentrated in homological degree $n$, with this $n$-th group being $\F$.

\subsection{Omitting generators of positive homological degree}
\!In \cite[Rm.3.3]{MR4244204}, it was observed that one can increase the efficiency of calculating the $s$-invariant by ignoring objects in homological degrees bigger than $1$. This can also be done in our situation, and in fact, we can further ignore generators in positive homological degree.

To see that this works, let $D$ be a cochain complex over $\F$ with a filtration
\[
0 \subset \cdots \subset\mathscr{F}_j(D) \subset \mathscr{F}_{j-2}(D) \subset \cdots \subset D,
\]
and let $D_+$ be the subcomplex of all elements of positive homological degree. Then both $D_+$ and $D/D_+$ admit a filtration and we have that $H^0(D)$ is a subspace of $H^0(D/D_+)$. Furthermore, $q$-gradings are defined as before for the vector spaces $H^0(D)$ and $H^0(D/D_+)$.

\begin{lemma}\label{lm:zerothhom}
Let $x\in H^0(D)$, and denote $p\colon H^0(D)\to H^0(D/D_+)$ the map induced by projection. Then $|x|_q = |p(x)|_q$, that is, the $q$-gradings of $x$ and $p(x)$ agree.
\end{lemma}

\begin{proof}
Notice that $\mathscr{F}^0_j(D) = \mathscr{F}^0_j(D/D_+)$, and $\mathscr{F}_j^1(D/D_+) = 0 = (D/D_+)^1$. Hence $|x|_q \leq |p(x)|_q$. Now if $c\in \mathscr{F}_j^0(D/D_+)$ represents $p(x)\in H^0(D/D_+)$, and $d\in D^0$ represents $x\in H^0(D)$, then $c-d\in (D/D_+)^0 = D^0$ is a coboundary. Hence $c-d = \partial(e)$ for some $e\in (D/D_+)^{-1} = D^{-1}$. Since $d$ is a cocycle in $D^0$, we get that $c$ is also a cocycle in $D^0$, and therefore $c\in \mathscr{F}_j^0(D)$ represents $x\in H^0(D)$. Hence $|p(x)|_q\leq |x|_q$.
\end{proof}

As in \cite[\S 7.2]{MR4244204}, we can throw away generators of positive homological degree, thus enabling us to calculate the $q$-degree of $\mathfrak{s}_\mathcal{O}$ as an element of $H^0(D/D_+)$.

\section{Computing $\mathfrak{sl}_3$-$s$-invariants}

Khovanov's $\mathfrak{sl}_3$-link homology \cite{MR2100691} gives rise to another family of $s$-invariants that can be computed with a Bar-Natan-style fast algorithm, see \cite{MR3248745, schuetz2025var}.

The definition requires a rank $3$-Frobenius system, which in the case of $\F_p$, the prime field of characteristic $p\not=3$ is given by
\[
\mathcal{F}_p = (\F_p, \F_p[X]/(X^3-1), \varepsilon, \Delta),
\]
where
\[
\Delta(1) = -X^2\otimes 1 - X\otimes X - 1\otimes X^2,
\]
and
\[
\varepsilon(1) = 0, \hspace{1cm} \varepsilon(X) = -1.
\]
For the prime field $\F_3$ of characteristic $3$, it is given by
\[
\mathcal{F}_3 = (\F_3, \F_3[X]/(X^3-X), \varepsilon, \Delta_3),
\]
where
\[
\Delta_3(1) = -X^2\otimes 1 - X\otimes X - 1\otimes X^2 +1\otimes 1,
\]
and $\varepsilon$ as in $\mathcal{F}_p$.

Unlike in Section \ref{sec:Frob}, rather than using smoothings of tangles, we need to resolve tangles into webs, and the role of dotted cobordisms is played by dotted foams. We will not give precise definitions for webs and foams, the details can be found in \cite{MR2100691, MR2336253, MR3248745, schuetz2025var}.

The analogue of $\Cobl^\F(B,\dot{B})$ is the category $\Fml^\F(B,\dot{B})$, where we use the following local relations on foams:
\begin{align}
\begin{tikzpicture}[baseline={([yshift=-.6ex]current bounding box.center)}]
\parallelbox{0}{0}{1}{0.66}
\node[scale = 0.75] at (0.6, 0.33) {$\bullet\bullet\bullet$};
\end{tikzpicture}
 &= 
\begin{tikzpicture}[baseline={([yshift=-.6ex]current bounding box.center)}]
\parallelbox{0}{0}{1}{0.66}
\node[scale = 0.75] at (0.6, 0.33) {$\ast$};
\end{tikzpicture}
\tag{3D}\label{eq:3d}\\[0.2cm]
- \,\,
\begin{tikzpicture}[baseline={([yshift=-.6ex]current bounding box.center)}]
\cylinder{0}{0}{1}{1.2}
\end{tikzpicture}
\,&= 
\begin{tikzpicture}[baseline={([yshift=-.6ex]current bounding box.center)}]
\bowl{0}{1.2}{1}{0.05}
\bowlud{0}{0}{1}{0.05}
\node[scale = 0.75] at (0.4, 0.9) {$\bullet\,\bullet$};
\end{tikzpicture}
+
\begin{tikzpicture}[baseline={([yshift=-.6ex]current bounding box.center)}]
\bowl{0}{1.2}{1}{0.05}
\bowlud{0}{0}{1}{0.05}
\node[scale = 0.75] at (0.4, 0.9) {$\bullet$};
\node[scale = 0.75] at (0.4, 0.1) {$\bullet$};
\end{tikzpicture}
+
\begin{tikzpicture}[baseline={([yshift=-.6ex]current bounding box.center)}]
\bowl{0}{1.2}{1}{0.05}
\bowlud{0}{0}{1}{0.05}
\node[scale = 0.75] at (0.4, 0.1) {$\bullet\,\bullet$};
\end{tikzpicture}
\,\, \left[\,\, -
\begin{tikzpicture}[baseline={([yshift=-.6ex]current bounding box.center)}]
\bowl{0}{1.2}{1}{0.05}
\bowlud{0}{0}{1}{0.05}
\end{tikzpicture}
\,\,\right]
\tag{CN} \label{eq:cn}\\[0.2cm]
\begin{tikzpicture}[baseline={([yshift=-.6ex]current bounding box.center)}]
\sphere{0}{0}{0.45}
\end{tikzpicture}
&= 
\begin{tikzpicture}[baseline={([yshift=-.6ex]current bounding box.center)}]
\sphere{0}{0}{0.45}
\node[scale = 0.75] at (0, 0.25) {$\bullet$};
\end{tikzpicture}
=0, \hspace{1cm}
\begin{tikzpicture}[baseline={([yshift=-.6ex]current bounding box.center)}]
\sphere{0}{0}{0.45}
\node[scale = 0.75] at (0, 0.25) {$\bullet\, \bullet$};
\end{tikzpicture}
= -1
\tag{S}\label{eq:s}
\end{align}
Here the $\ast$ in (\ref{eq:3d}) is a dot if we work in characteristic $3$, and it is empty if not. Also, the term in brackets in (\ref{eq:cn}) is only present in characteristic $3$.

With these relations any closed foam can be reduced to a disjoint union of {\em theta foams}, which are evaluated as in Figure \ref{fig:theta_eval}.

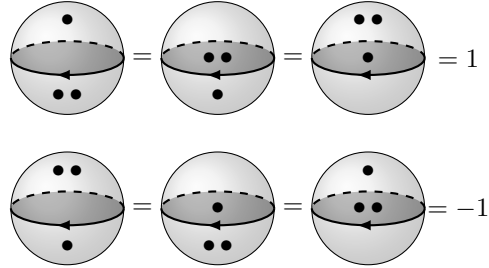
\begin{figure}[ht]
\[
\begin{tikzpicture}
\thetafoam{0}{0}{0.75}
\node at (0, 0.5) {$\bullet$};
\node at (0, -0.5) {$\bullet \, \bullet$};
\node at (1,0) {$=$};
\thetafoam{2}{0}{0.75}
\node at (2,0) {$\bullet\, \bullet$};
\node at (2, -0.5) {$\bullet$};
\node at (3,0) {$=$};
\thetafoam{4}{0}{0.75}
\node at (4,0.5) {$\bullet\,\bullet$};
\node at (4,0) {$\bullet$};
\node at (5.2,0) {$=1$};
\thetafoam{0}{-2}{0.75}
\node at (0, -2.5) {$\bullet$};
\node at (0, -1.5) {$\bullet \, \bullet$};
\node at (1,-2) {$=$};
\thetafoam{2}{-2}{0.75}
\node at (2,-2.5) {$\bullet\, \bullet$};
\node at (2, -2) {$\bullet$};
\node at (3,-2) {$=$};
\thetafoam{4}{-2}{0.75}
\node at (4,-2) {$\bullet\,\bullet$};
\node at (4,-1.5) {$\bullet$};
\node at (5.2,-2) {$=-1$};
\end{tikzpicture}
\]
\caption{\label{fig:theta_eval}The evaluation of theta foams. Any other combination of dots with at most two dots on any disc evaluates as $0$.}
\end{figure}

If $B = \emptyset$, this set of relations is sufficient, otherwise a few more local relations are needed, compare \cite[\S 7]{schuetz2025var}.

Given a tangle $T$ consisting of an oriented crossing in a disc $B$ with $4$ points $\dot{B}$, we can obtain a cochain complex $C_{f(X)}(T;\F)$ as in Figure \ref{fig:crossing_complex_sl3}, where $f(X)$ is either the polynomial $X^3-X$ or $X^3-1$, depending on whether the characteristic of $\F$ is $3$ or not.

\begin{figure}[ht]
\[
\begin{tikzpicture}
\pcrossing{0}{2}{1}
\mcrossing{0}{0}{1}
\node at (1.5, 2.5) {$:$};
\node at (2.5, 2.5) {$u^{-1}q^3$};
%\smoothing{3}{1}{1}{->}
%\spidering{-3}{1}{1}{->}
\spidering{3}{2}{1}{->}
\draw[->] (4.2, 2.5) -- node [above] {$U$} (5.3, 2.5);
\node at (6, 2.5) {$u^0q^2$};
\nssmoothing{6.5}{2}{1}{->}
\node at (1.5, 0.5) {$\colon$};
\node at (2.5, 0.5) {$u^0q^{-2}$};
\nssmoothing{3}{0}{1}{->}
\draw[->] (4.2, 0.5) -- node [above] {$Z$} (5.3, 0.5);
\node at (6, 0.5) {$u^1q^{-3}$};
\spidering{6.5}{0}{1}{->}
\end{tikzpicture}
\]
\caption{\label{fig:crossing_complex_sl3}The complexes associated to a positive and a negative crossing. Here $U$ and $Z$ stand for the unzip and zip foams.}
\end{figure}
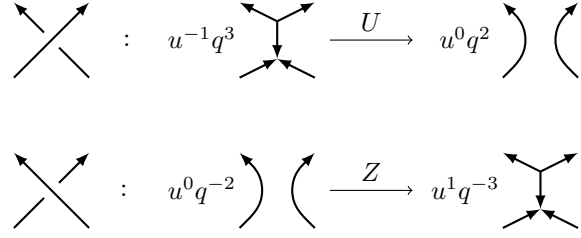

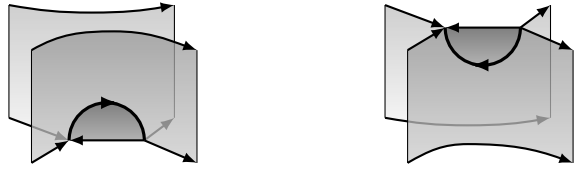
\begin{figure}[ht]
\[
\begin{tikzpicture}
\shade[top color = gray!50, bottom color = gray!10, opacity = 0.75] (-0.3, 0.6) -- (0.5, 0.3) to [out = 90, in = 180] (1, 0.8) to [out = 0, in = 90] (1.5, 0.3) -- (1.9, 0.6) -- (1.9, 2.1) to [out = 190, in = 0] (0.8, 2) to [out = 180, in = 350] (-0.3, 2.1) -- (-0.3, 0.6);
\draw[-] (1.9, 0.6) -- (1.9, 2.1);
\draw[-] (-0.3, 0.6) -- (-0.3, 2.1);
\shade[top color = gray, bottom color = gray!60] (0.5, 0.3) to [out = 90, in = 180] (1, 0.8) to [out = 0, in = 90] (1.5, 0.3) -- (0.5, 0.3);
\draw[->, thick, >=latex] (-0.3, 0.6) -- (0.5, 0.3);
\draw[->, thick, >=latex] (1.5, 0.3) -- (1.9, 0.6);
\shade[top color = gray!80, bottom color = gray!40, opacity = 0.75] (0,0) -- (0.5, 0.3) to [out = 90, in = 180] (1, 0.8) to [out = 0, in = 90] (1.5, 0.3) -- (2.2, 0) -- (2.2, 1.5) to [out = 160, in = 0] (1.1, 1.75) to [out = 180, in = 30] (0, 1.5) -- (0,0);
\draw[-] (0,0) -- (0, 1.5);
\draw[-] (2.2, 0) -- (2.2, 1.5);
\draw[->, >=latex, thick] (0,0) -- (0.5, 0.3);
\draw[<-, >=latex, thick] (0.5, 0.3) -- (1.5, 0.3);
\draw[->, >=latex, thick] (1.5, 0.3) -- (2.2, 0);
\draw[-, very thick] (0.5, 0.3) to [out = 90, in = 180] (1, 0.8) to [out = 0, in = 90] (1.5, 0.3);
\draw[->, very thick, >=latex] (1.1,0.8) -- (1.15, 0.8);
\draw[->, >=latex, thick] (0, 1.5) to [out = 30, in = 180] (1.1, 1.75) to [out = 0, in = 160] (2.2, 1.5);
\draw[->, >= latex, thick] (-0.3, 2.1) to [out = 350, in = 180] (0.8, 2) to [out = 0, in = 190] (1.9, 2.1);
\shade[top color = gray!50, bottom color = gray!10, opacity = 0.75] (4.7, 2.1) -- (5.5, 1.8) to [out= 270, in = 180] (6, 1.3) to [out=0, in = 270] (6.5, 1.8) -- (6.9, 2.1) -- (6.9, 0.6) to [out = 190, in = 0] (5.8, 0.5) to [out = 180, in = 350] (4.7, 0.6) -- (4.7, 2.1);
\draw[->, >=latex, thick] (4.7, 0.6) to [out = 350, in = 180] (5.8, 0.5) to [out = 0, in = 190] (6.9, 0.6);
\draw[-] (4.7, 2.1) -- (4.7, 0.6);
\draw[-] (6.9, 0.6) -- (6.9, 2.1);
\shade[top color = gray, bottom color = gray!60] (5.5, 1.8) to [out = 270, in = 180] (6, 1.3) to [out = 0, in = 270] (6.5, 1.8);
\shade[top color = gray!80, bottom color = gray!40, opacity = 0.75] (5, 0) to [out = 30, in = 180] (5.8, 0.25) to [out = 0, in = 160] (7.2, 0) -- (7.2, 1.5) -- (6.5, 1.8) to [out = 270, in = 0] (6, 1.3) to [out = 180, in = 270] (5.5, 1.8) -- (5, 1.5) -- (5, 0);
\draw[->, >=latex, thick] (5, 0) to [out = 30, in = 180] (5.8, 0.25) to [out = 0, in = 160] (7.2, 0);
\draw[->, >=latex, thick] (5, 1.5) -- (5.5, 1.8);
\draw[<-, >=latex, thick] (5.5, 1.8) -- (6.5, 1.8);
\draw[->, >=latex, thick] (6.5, 1.8) -- (7.2, 1.5);
\draw[->, >=latex, thick] (6.5, 1.8) -- (6.9, 2.1);
\draw[->, >=latex, thick] (4.7, 2.1) -- (5.5, 1.8);
\draw[-, very thick] (5.5, 1.8) to [out = 270, in = 180] (6, 1.3) to [out = 0, in = 270] (6.5, 1.8);
\draw[<-, >=latex, very thick] (5.85, 1.3) -- (5.9, 1.3);
\draw[-] (5, 0) -- (5, 1.5);
\draw[-] (7.2, 0) -- (7.2, 1.5);
\end{tikzpicture}
\]
\caption{\label{fig:unzip_foams}The unzip and zip foams.}
\end{figure}

For an oriented link diagram $D$ representing $L$, we also get a cochain complex $C_{f(X)}(D;\F)$, whose total homology is of dimension $3^{|L|}$, since the polynomials $X^3-X$ and $X^3-1$ are separable after passing to an algebraic closure, see \cite[\S 3]{MR2336253}, where explicit generators are given.

Furthermore, $H^0_{f(X)}(L;\F)$ is at least $3$-dimensional, and the generators playing the role of $\mathfrak{s}_\mathcal{O}$ from Section \ref{sec:Frob} can be described as follows: By using the smoothing in homological degree $0$ for each crossing, compare Figure \ref{fig:crossing_complex_sl3}, we get a web $S$ which simply consists of $k$ oriented circles. This web spans a subspace of $C^0_{f(X)}(D;\F)$ of the form $(\F[X]/(f(X)))^{\otimes k}$.

Each root $r$ gives rise to an idempotent element $I_r\in \F[X]/(f(X))$, and the element
\[
I_r(L) = I_r \otimes I_r \otimes \cdots \otimes I_r \in (\F[X]/(f(X)))^{\otimes k}
\]
is a cocycle which is a generator of $H^0_{f(X)}(L;\F)$, see \cite[\S 3]{MR2336253}.

We have that $1$ is a root of both $X^3-X$ and $X^3-1$, and the element $I_1$ is given as
\[
I_1 = \left\{ \begin{array}{cc}
\frac{1}{3}\left(X^2+X+1\right) & \chr\F\not=3, \\[0.2cm]
-X^2-X & \chr\F=3.
\end{array}
\right.
\]
Notice that $I_1$ is defined over the prime field, even if $f(X)$ does not factor over the prime field.

The complex $C_{f(X)}(D;\F)$ admits a filtration, which is due to a change in convention in \cite{MR2100691} increasing:
\[
0\subset \cdots \subset \mathscr{F}_j \subset \mathscr{F}_{j+2} \subset \cdots \subset C_{f(X)}(D;\F)
\]
In particular, non-zero elements of $H_{f(X)}(L;\F)$ have a $q$-grading, and the analogue of Definition \ref{def:s-inv} is given as follows.

\begin{definition}\label{def:sl3-s-inv}
Let $L$ be an oriented link and $\F$ a field of characteristic $p$. The {\em $\mathfrak{sl}_3$-$s$-invariant } of $L$ is defined as
\[
s^p_{\mathfrak{sl}_3}(L) = - \left(\frac{|I_1(\overline{L})|_q-2}{2}\right),
\]
with $\overline{L}$ the mirror of $L$.
\end{definition}

This slightly cumbersome definition involving the mirror of $L$ and the minus sign is because of the aforementioned convention change. This definition ensures that $s^p_{\mathfrak{sl}_3}(L) = s^p(L)$ for $L$ an unlink or a Hopf link. It also agrees with the definition for a knot, which follows from \cite[\S 6]{schuetz2025var}.

Lemma \ref{lm:basic_prop} and \ref{lm:slice_obst} carry over to $s^p_{\mathfrak{sl}_3}(L)$. In particular,  we also get a weak-slice obstruction.

The filtered complex $C_{f(X)}(D;\F)$ admits a fast scanning approach, see \cite{schuetz2025var}, and the method to compute $s^p(L)$ from Section \ref{sec:fast_comp} carries over to $s^p_{\mathfrak{sl}_3}(L)$. For a positive or negative crossing $T_\pm$ we need to describe the cochain maps $\varphi_\pm \colon C_{T_\pm}\to C_{f(X)}(T_\pm;\F)$. It has to mimic $I_1\otimes I_1$, so it is of the form
\begin{align*}
\varphi_\pm & = \frac{1}{9} \left(\,\,\,
\begin{tikzpicture}[baseline={([yshift=-.6ex]current bounding box.center)}]
\smoothing{0}{0}{0.8}{-}
\node[scale=0.7] at (0.15,0.3) {$\bullet$};
\node[scale=0.7] at (0.15,0.5) {$\bullet$};
\node[scale=0.7] at (0.65,0.3) {$\bullet$};
\node[scale=0.7] at (0.65,0.5) {$\bullet$};
\end{tikzpicture}
+
\begin{tikzpicture}[baseline={([yshift=-.6ex]current bounding box.center)}]
\smoothing{0}{0}{0.8}{-}
\node[scale=0.7] at (0.15,0.3) {$\bullet$};
\node[scale=0.7] at (0.15,0.5) {$\bullet$};
\node[scale=0.7] at (0.635, 0.4) {$\bullet$};
\end{tikzpicture}
+
\begin{tikzpicture}[baseline={([yshift=-.6ex]current bounding box.center)}]
\smoothing{0}{0}{0.8}{-}
\node[scale=0.7] at (0.15,0.3) {$\bullet$};
\node[scale=0.7] at (0.15,0.5) {$\bullet$};
\end{tikzpicture}
\right. \\
& \hskip0.865cm+
\begin{tikzpicture}[baseline={([yshift=-.6ex]current bounding box.center)}]
\smoothing{0}{0}{0.8}{-}
\node[scale=0.7] at (0.165,0.4) {$\bullet$};
\node[scale=0.7] at (0.65,0.3) {$\bullet$};
\node[scale=0.7] at (0.65,0.5) {$\bullet$};
\end{tikzpicture}
+
\begin{tikzpicture}[baseline={([yshift=-.6ex]current bounding box.center)}]
\smoothing{0}{0}{0.8}{-}
\node[scale=0.7] at (0.165,0.4) {$\bullet$};
\node[scale=0.7] at (0.635, 0.4) {$\bullet$};
\end{tikzpicture}
+
\begin{tikzpicture}[baseline={([yshift=-.6ex]current bounding box.center)}]
\smoothing{0}{0}{0.8}{-}
\node[scale=0.7] at (0.165,0.4) {$\bullet$};
\end{tikzpicture}
\\
& \left. \hskip0.935cm+
\begin{tikzpicture}[baseline={([yshift=-.6ex]current bounding box.center)}]
\smoothing{0}{0}{0.8}{-}
\node[scale=0.7] at (0.65,0.3) {$\bullet$};
\node[scale=0.7] at (0.65,0.5) {$\bullet$};
\end{tikzpicture}
+
\begin{tikzpicture}[baseline={([yshift=-.6ex]current bounding box.center)}]
\smoothing{0}{0}{0.8}{-}
\node[scale=0.7] at (0.635, 0.4) {$\bullet$};
\end{tikzpicture}
+
\begin{tikzpicture}[baseline={([yshift=-.6ex]current bounding box.center)}]
\smoothing{0}{0}{0.8}{-}
\end{tikzpicture}
\,\,\,
\right)
\end{align*}
if the characteristic is different from $3$, and
\[
\varphi_\pm = -
\begin{tikzpicture}[baseline={([yshift=-.6ex]current bounding box.center)}]
\smoothing{0}{0}{0.8}{-}
\node[scale=0.7] at (0.15,0.3) {$\bullet$};
\node[scale=0.7] at (0.15,0.5) {$\bullet$};
\node[scale=0.7] at (0.65,0.3) {$\bullet$};
\node[scale=0.7] at (0.65,0.5) {$\bullet$};
\end{tikzpicture}
-
\begin{tikzpicture}[baseline={([yshift=-.6ex]current bounding box.center)}]
\smoothing{0}{0}{0.8}{-}
\node[scale=0.7] at (0.15,0.3) {$\bullet$};
\node[scale=0.7] at (0.15,0.5) {$\bullet$};
\node[scale=0.7] at (0.635, 0.4) {$\bullet$};
\end{tikzpicture}
-
\begin{tikzpicture}[baseline={([yshift=-.6ex]current bounding box.center)}]
\smoothing{0}{0}{0.8}{-}
\node[scale=0.7] at (0.165,0.4) {$\bullet$};
\node[scale=0.7] at (0.65,0.3) {$\bullet$};
\node[scale=0.7] at (0.65,0.5) {$\bullet$};
\end{tikzpicture}
-
\begin{tikzpicture}[baseline={([yshift=-.6ex]current bounding box.center)}]
\smoothing{0}{0}{0.8}{-}
\node[scale=0.7] at (0.165,0.4) {$\bullet$};
\node[scale=0.7] at (0.635,0.4) {$\bullet$};
\end{tikzpicture}
\]
if the characteristic is $3$. Since $I_1$ is an idempotent, tensoring the $\varphi_\pm$ combines to a cochain map $\varphi\colon (\Z[X]/(f(X)))^{3^k}\to C_{f(X)}(D;\F)$ for an oriented link diagram $D$ which smoothes to a web consisting of $k$ circles, and so that $\varphi(1\otimes \cdots \otimes 1) = I_1(L)$.

We can mimic the scanning algorithm from Section \ref{sec:fast_comp} to get a cochain map $\varphi\colon u^0\F \to D$ where $D$ is filtered chain homotopic to $C_{f(X)}(D;\F)$ and $\varphi(1)$ representing $I_1(L)$.

\section{Computations}

Both $s^p$ and $s^p_{\mathfrak{sl}_3}$ are slice-torus link invariants in the sense of \cite{MR4176697} (up to a factor), so by \cite[Rm.2.6]{MR4311821} they agree with the signature for non-split alternating links. For $s^p$ this already follows from \cite{MR2509750}, but the $E_\infty$-page of the $\mathfrak{sl}_3$-spectral sequence can have more than three non-trivial groups $E_\infty^{0,j}$ for alternating links (the Borromean rings being an example). We will therefore restrict our attention to non-alternating knots and links.

The algorithms to compute $s^p$ and $s^p_{\mathfrak{sl}_3}$ have been implemented in \verb+knotjob+, which is available on the author's website at \url{https://www.maths.dur.ac.uk/users/dirk.schuetz/knotjob.html}.

The $2$-component link $L_\varepsilon$ in Figure \ref{fig:int_link} has linking number $0$, but the $s$-invariant changes if we change the orientation on a component. We have $s(L_+)=1$ and $s(L_-)=3$, independent of characteristic and also if we consider $s_{\mathfrak{sl}_3}$. The signature of this link is $3$, independent of the orientation. This implies Theorem \ref{th:nonconst}.

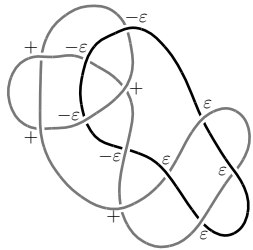
\begin{figure}[ht]
\[
\begin{tikzpicture}
\tikzstyle{box}=[-, line width = 1pt]
\tikzstyle{bigbox}=[-, line width = 3pt]
\colorlet{colr-1}{white}
\definecolor{colr0}{rgb}{0.0, 0.0, 0.0}
\definecolor{colr1}{rgb}{0.47058824, 0.47058824, 0.47058824}
\draw[box, colr1] (3.5199835, -2.4427867) .. controls (3.4453409, -2.5755658) and (3.3698187, -2.6402209) .. (3.2834275, -2.765384);
\draw[bigbox, colr-1] (3.2842364, -2.4325337) .. controls (3.370492, -2.5627804) and (3.4448147, -2.6333933) .. (3.5143614, -2.7732873);
\draw[box, colr0] (3.2842364, -2.4325337) .. controls (3.370492, -2.5627804) and (3.4448147, -2.6333933) .. (3.5143614, -2.7732873);
\draw[box, colr0] (3.1036804, -3.391436) .. controls (2.9900649, -3.3187697) and (2.930367, -3.2297332) .. (2.8426797, -3.1117072);
\draw[bigbox, colr-1] (2.8335001, -3.4242635) .. controls (2.9394262, -3.319032) and (2.9923658, -3.221788) .. (3.0760872, -3.0905328);
\draw[box, colr1] (2.8335001, -3.4242635) .. controls (2.9394262, -3.319032) and (2.9923658, -3.221788) .. (3.0760872, -3.0905328);
\draw[box, colr1] (2.3776698, -2.805607) .. controls (2.4882326, -2.6945357) and (2.5486743, -2.5972965) .. (2.6332238, -2.4481168);
\draw[bigbox, colr-1] (2.6333146, -2.804327) .. controls (2.547806, -2.6815357) and (2.4925532, -2.598876) .. (2.3952565, -2.511337);
\draw[box, colr0] (2.6333146, -2.804327) .. controls (2.547806, -2.6815357) and (2.4925532, -2.598876) .. (2.3952565, -2.511337);
\draw[box, colr0] (2.1237628, -2.357427) .. controls (2.0045283, -2.3125892) and (1.9210886, -2.2989554) .. (1.7984768, -2.2609444);
\draw[bigbox, colr-1] (1.9174112, -2.5148077) .. controls (1.9492264, -2.3557985) and (1.9814459, -2.2524517) .. (2.015014, -2.0906312);
\draw[box, colr1] (1.9174112, -2.5148077) .. controls (1.9492264, -2.3557985) and (1.9814459, -2.2524517) .. (2.015014, -2.0906312);
\draw[box, colr0] (1.5218171, -2.0974092) .. controls (1.4339882, -1.9830371) and (1.3980212, -1.8692707) .. (1.3702894, -1.7126559);
\draw[bigbox, colr-1] (1.2703313, -1.990307) .. controls (1.390427, -1.9487829) and (1.4631568, -1.8949437) .. (1.5697972, -1.824806);
\draw[box, colr1] (1.2703313, -1.990307) .. controls (1.390427, -1.9487829) and (1.4631568, -1.8949437) .. (1.5697972, -1.824806);
\draw[box, colr1] (1.2821871, -1.0638852) .. controls (1.4028058, -1.0860723) and (1.4749503, -1.1413848) .. (1.5799369, -1.2020048);
\draw[bigbox, colr-1] (1.3826376, -1.3251431) .. controls (1.4158235, -1.1680264) and (1.4493293, -1.0522244) .. (1.5414597, -0.9386938);
\draw[box, colr0] (1.3826376, -1.3251431) .. controls (1.4158235, -1.1680264) and (1.4493293, -1.0522244) .. (1.5414597, -0.9386938);
\draw[box, colr1] (1.8417004, -0.5085428) .. controls (1.9578016, -0.6123165) and (1.9923549, -0.73954046) .. (2.0212848, -0.8977376);
\draw[bigbox, colr-1] (1.8217784, -0.7610794) .. controls (1.9361289, -0.70696014) and (2.0288923, -0.6636008) .. (2.1687918, -0.70829624);
\draw[box, colr0] (1.8217784, -0.7610794) .. controls (1.9361289, -0.70696014) and (2.0288923, -0.6636008) .. (2.1687918, -0.70829624);
\draw[box, colr0] (2.8817933, -1.6855674) .. controls (2.9487748, -1.8510147) and (2.9957967, -1.9544821) .. (3.0763218, -2.093406);
\draw[bigbox, colr-1] (3.1233373, -1.8055154) .. controls (3.0034866, -1.8608841) and (2.931069, -1.9404411) .. (2.8408418, -2.0752172);
\draw[box, colr1] (3.1233373, -1.8055154) .. controls (3.0034866, -1.8608841) and (2.931069, -1.9404411) .. (2.8408418, -2.0752172);
\draw[box, colr1] (1.9782424, -3.30526) .. controls (1.9021413, -3.166286) and (1.8818872, -3.0500443) .. (1.8757832, -2.8993282);
\draw[bigbox, colr-1] (1.7036158, -3.091093) .. controls (1.8475343, -3.1183643) and (1.9459971, -3.0920715) .. (2.0745745, -3.0266683);
\draw[box, colr1] (1.7036158, -3.091093) .. controls (1.8475343, -3.1183643) and (1.9459971, -3.0920715) .. (2.0745745, -3.0266683);
\draw[box, colr1] (2.0345728, -1.6876012) .. controls (1.9947124, -1.5344985) and (1.9229678, -1.4553958) .. (1.8286486, -1.3673704);
\draw[bigbox, colr-1] (1.8231711, -1.6342928) .. controls (1.9207629, -1.5404361) and (1.9936926, -1.4598954) .. (2.0337424, -1.303325);
\draw[box, colr1] (1.8231711, -1.6342928) .. controls (1.9207629, -1.5404361) and (1.9936926, -1.4598954) .. (2.0337424, -1.303325);
\draw[box, colr1] (0.8284877, -1.3190283) .. controls (0.83377683, -1.1342229) and (0.83205503, -0.9974594) .. (0.9026395, -0.83475536);
\draw[bigbox, colr-1] (0.9882927, -1.0786452) .. controls (0.87470496, -1.0819458) and (0.7873123, -1.0504074) .. (0.6668078, -1.0981022);
\draw[box, colr1] (0.9882927, -1.0786452) .. controls (0.87470496, -1.0819458) and (0.7873123, -1.0504074) .. (0.6668078, -1.0981022);
\draw[box, colr1] (0.6524646, -1.995123) .. controls (0.7710768, -2.0514514) and (0.859832, -2.0322232) .. (0.97542685, -2.027552);
\draw[bigbox, colr-1] (0.8207397, -1.7859881) .. controls (0.82010555, -1.971296) and (0.8171699, -2.103866) .. (0.87309, -2.2788036);
\draw[box, colr1] (0.8207397, -1.7859881) .. controls (0.82010555, -1.971296) and (0.8171699, -2.103866) .. (0.87309, -2.2788036);
\draw[box, colr0] (3.2842364, -2.4325337) .. controls (3.2029235, -2.3097503) and (3.1518705, -2.2237446) .. (3.0763218, -2.093406);
\draw[box, colr1] (3.6, -2.0771747) .. controls (3.5805633, -1.9436721) and (3.5193653, -1.8506097) .. (3.421986, -1.7978032);
\draw[box, colr0] (3.5809755, -3.148485) .. controls (3.5581026, -3.28236) and (3.4963248, -3.3755765) .. (3.3978806, -3.424496);
\draw[box, colr1] (3.2834275, -2.765384) .. controls (3.202295, -2.8829281) and (3.1540394, -2.968322) .. (3.0760872, -3.0905328);
\draw[box, colr1] (2.5416012, -3.5923524) .. controls (2.4270792, -3.6172998) and (2.3418922, -3.608127) .. (2.2333295, -3.5545473);
\draw[box, colr0] (2.8426797, -3.1117072) .. controls (2.76058, -3.001202) and (2.7129123, -2.9186301) .. (2.6333146, -2.804327);
\draw[box, colr1] (2.3776698, -2.805607) .. controls (2.2744074, -2.9093444) and (2.195741, -2.9650345) .. (2.0745745, -3.0266683);
\draw[box, colr0] (2.3952565, -2.511337) .. controls (2.3043613, -2.4295578) and (2.231436, -2.3979168) .. (2.1237628, -2.357427);
\draw[box, colr1] (2.6332238, -2.4481168) .. controls (2.7123373, -2.3085284) and (2.7542832, -2.2045133) .. (2.8408418, -2.0752172);
\draw[box, colr1] (1.9174112, -2.5148077) .. controls (1.8888249, -2.6576793) and (1.8699555, -2.7554336) .. (1.8757832, -2.8993282);
\draw[box, colr0] (1.7984768, -2.2609444) .. controls (1.6861757, -2.2261295) and (1.6016195, -2.201329) .. (1.5218171, -2.0974092);
\draw[box, colr1] (2.015014, -2.0906312) .. controls (2.04581, -1.9421736) and (2.0712829, -1.8286023) .. (2.0345728, -1.6876012);
\draw[box, colr1] (1.2703313, -1.990307) .. controls (1.1625403, -2.0275767) and (1.0859637, -2.0230849) .. (0.97542685, -2.027552);
\draw[box, colr0] (1.3702894, -1.7126559) .. controls (1.3452052, -1.5709922) and (1.3528097, -1.4663616) .. (1.3826376, -1.3251431);
\draw[box, colr1] (1.5697972, -1.824806) .. controls (1.6675313, -1.7605258) and (1.7347997, -1.7192822) .. (1.8231711, -1.6342928);
\draw[box, colr1] (1.2821871, -1.0638852) .. controls (1.1731972, -1.0438371) and (1.0977825, -1.0754638) .. (0.9882927, -1.0786452);
\draw[box, colr0] (1.5414597, -0.9386938) .. controls (1.6253953, -0.83526164) and (1.7115943, -0.8132268) .. (1.8217784, -0.7610794);
\draw[box, colr1] (1.5799369, -1.2020048) .. controls (1.6756561, -1.2572738) and (1.7435337, -1.2879351) .. (1.8286486, -1.3673704);
\draw[box, colr1] (1.4968807, -0.4) .. controls (1.3667628, -0.40610734) and (1.2765057, -0.44021103) .. (1.1623862, -0.5214318);
\draw[box, colr0] (2.4763622, -0.9410722) .. controls (2.5750408, -1.052467) and (2.630609, -1.1436664) .. (2.7048433, -1.2838165);
\draw[box, colr1] (2.0212848, -0.8977376) .. controls (2.0485082, -1.0466042) and (2.0703068, -1.1603801) .. (2.0337424, -1.303325);
\draw[box, colr1] (1.3650775, -2.9331543) .. controls (1.2488841, -2.8512459) and (1.175483, -2.781456) .. (1.0819257, -2.6600246);
\draw[box, colr1] (0.8284877, -1.3190283) .. controls (0.82348067, -1.4939741) and (0.8213393, -1.6107482) .. (0.8207397, -1.7859881);
\draw[box, colr1] (0.43495607, -1.3493638) .. controls (0.3871246, -1.4772533) and (0.38553494, -1.5980356) .. (0.42932665, -1.7285775);
\draw[box, colr1] (3.421986, -1.7978032) .. controls (3.3221748, -1.7436777) and (3.232485, -1.7550912) .. (3.1233373, -1.8055154);
\draw[box, colr1] (3.6, -2.0771747) .. controls (3.6197352, -2.2127273) and (3.5870597, -2.3234673) .. (3.5199835, -2.4427867);
\draw[box, colr0] (3.3978806, -3.424496) .. controls (3.2975893, -3.474333) and (3.2063267, -3.4570868) .. (3.1036804, -3.391436);
\draw[box, colr0] (3.5809755, -3.148485) .. controls (3.604278, -3.0120962) and (3.5767903, -2.8988633) .. (3.5143614, -2.7732873);
\draw[box, colr1] (2.2333295, -3.5545473) .. controls (2.123274, -3.5002308) and (2.0486887, -3.433907) .. (1.9782424, -3.30526);
\draw[box, colr1] (2.5416012, -3.5923524) .. controls (2.657505, -3.5671039) and (2.7365546, -3.5205736) .. (2.8335001, -3.4242635);
\draw[box, colr1] (1.1623862, -0.5214318) .. controls (1.0471816, -0.6034248) and (0.96837056, -0.6832388) .. (0.9026395, -0.83475536);
\draw[box, colr1] (1.4968807, -0.4) .. controls (1.6281595, -0.39383817) and (1.7315513, -0.41008902) .. (1.8417004, -0.5085428);
\draw[box, colr0] (2.7048433, -1.2838165) .. controls (2.7796266, -1.4250033) and (2.818821, -1.5300227) .. (2.8817933, -1.6855674);
\draw[box, colr0] (2.4763622, -0.9410722) .. controls (2.3763502, -0.8281721) and (2.298524, -0.7497434) .. (2.1687918, -0.70829624);
\draw[box, colr1] (1.0819257, -2.6600246) .. controls (0.9875035, -2.5374706) and (0.92533475, -2.4422429) .. (0.87309, -2.2788036);
\draw[box, colr1] (1.3650775, -2.9331543) .. controls (1.482287, -3.0157793) and (1.5689888, -3.0655825) .. (1.7036158, -3.091093);
\draw[box, colr1] (0.42932665, -1.7285775) .. controls (0.4737997, -1.8611505) and (0.54337895, -1.9433188) .. (0.6524646, -1.995123);
\draw[box, colr1] (0.43495607, -1.3493638) .. controls (0.4834531, -1.2196949) and (0.5561483, -1.1419003) .. (0.6668078, -1.0981022);
\node[scale=0.7] at (0.7, -0.95) {$+$};
\node[scale=0.7] at (0.7, -2.15) {$+$};
\node[scale=0.7] at (2.1, -1.5) {$+$};
\node[scale=0.7] at (1.8, -3.2) {$+$};
\node[scale=0.7] at (2.1, -0.55) {$-\varepsilon$};
\node[scale=0.7] at (1.3,-0.95) {$-\varepsilon$};
\node[scale=0.7] at (1.2, -1.85) {$-\varepsilon$};
\node[scale=0.7] at (1.75,-2.4) {$-\varepsilon$};
\node[scale=0.7] at (2.5, -2.45) {$\varepsilon$};
\node[scale=0.7] at (3.25, -2.6) {$\varepsilon$};
\node[scale=0.7] at (3.05, -1.7) {$\varepsilon$};
\node[scale=0.7] at (3, -3.45) {$\varepsilon$};
\end{tikzpicture}
\]
\caption{\label{fig:int_link}The $2$-component link $L_\varepsilon$ with $\varepsilon\in \{\pm 1\}$ and oriented so to give the indicated crossing signs.}
\end{figure}

In particular, the $s$-invariant is not simply given by $s^p_{\min}(L)$ in general. The link in Figure \ref{fig:int_link} has $12$ crossings. We did not find any $2$-component links with up to $11$ crossings with this property. A list with these links can be found on the {\em Knot Atlas}, see \url{https://katlas.org/wiki/Main_Page}. We did find links with $11$ crossings and more than $2$ components, where $s^p(L)>s^p_{\min}(L)$. Since we did not check every possible orientation for links with more than two components, there may indeed be more.

During the Bar-Natan algorithm we have to keep track of one extra generator and how it maps to the cochain complex. This does not affect the computational complexity or computation time significantly. Also, recall that to calculate the $s$-invariant of a knot, we can ignore all generators of homological degree bigger than $1$. By keeping track of the cochain map, we can ignore all generators of positive homological degree. When applied to knots, we observed a slight improvement in performance. 

Despite these slight improvements, our computer programme still uses the old method when calculating $s$-invariants of knots. When applying the new method to knots, the $0$-th homology group need not be of dimension $2$, compare Lemma \ref{lm:zerothhom}. This removes one check for bugs in computer programmes. We have compared both methods for knots with up to $15$ crossings to check both give the same results. Our programme also checks that the image $\varphi\colon \F \to E^{0,\ast}_i$ represents a cocycle at every stage of the spectral sequence.

\bibliography{KnotHomology}

\providecommand{\bysame}{\leavevmode\hbox to3em{\hrulefill}\thinspace}
\providecommand{\MR}{\relax\ifhmode\unskip\space\fi MR }
\providecommand{\MRhref}[2]{%
  \href{http://www.ams.org/mathscinet-getitem?mr=#1}{#2}
}
\providecommand{\href}[2]{#2}
\begin{thebibliography}{MMSW23}

\bibitem[BN05]{MR2174270}
Dror Bar-Natan, \emph{Khovanov's homology for tangles and cobordisms}, Geom.
  Topol. \textbf{9} (2005), 1443--1499. \MR{2174270}

\bibitem[BN07]{MR2320156}
{D}ror Bar-Natan, \emph{Fast {K}hovanov homology computations}, J. Knot Theory
  Ramifications \textbf{16} (2007), no.~3, 243--255. \MR{2320156}

\bibitem[BW08]{MR2462446}
Anna Beliakova and Stephan Wehrli, \emph{Categorification of the colored
  {J}ones polynomial and {R}asmussen invariant of links}, Canad. J. Math.
  \textbf{60} (2008), no.~6, 1240--1266. \MR{2462446}

\bibitem[CC20]{MR4176697}
Alberto Cavallo and Carlo Collari, \emph{Slice-torus concordance invariants and
  {W}hitehead doubles of links}, Canad. J. Math. \textbf{72} (2020), no.~6,
  1423--1462. \MR{4176697}

\bibitem[Col21]{MR4311821}
Carlo Collari, \emph{Slice-torus link invariants, combinatorial invariants and
  positivity conditions}, Bull. Lond. Math. Soc. \textbf{53} (2021), no.~4,
  1072--1092. \MR{4311821}

\bibitem[Kho00]{MR1740682}
Mikhail Khova{n}ov, \emph{A categorification of the {J}ones polynomial}, Duke
  Math. J. \textbf{101} (2000), no.~3, 359--426. \MR{1740682}

\bibitem[K{h}o04]{MR2100691}
Mikhail K{h}ovanov, \emph{sl(3) link homology}, Algebr. Geom. Topol. \textbf{4}
  (2004), 1045--1081. \MR{2100691}

\bibitem[{K}ho06]{MR2232858}
Mikhail {K}hovanov, \emph{Link homology and {F}robenius extensions}, Fund.
  Math. \textbf{190} (2006), 179--190. \MR{2232858}

\bibitem[Lew13]{MR3248745}
Lukas Lewark, \emph{{$\mathfrak{sl}_3$}-foam homology calculations}, Algebr.
  Geom. Topol. \textbf{13} (2013), no.~6, 3661--3686. \MR{3248745}

\bibitem[MMSW23]{MR4541332}
Ciprian Manolescu, Marco Marengon, Sucharit Sarkar, and Michael Willis, \emph{A
  generalization of {R}asmussen's invariant, with applications to surfaces in
  some four-manifolds}, Duke Math. J. \textbf{172} (2023), no.~2, 231--311.
  \MR{4541332}

\bibitem[MO08]{MR2509750}
Ciprian Manolescu and Peter Ozsv\'{a}th, \emph{On the {K}hovanov and knot
  {F}loer homologies of quasi-alternating links}, Proceedings of {G}\"{o}kova
  {G}eometry-{T}opology {C}onference 2007, G\"{o}kova Geometry/Topology
  Conference (GGT), G\"{o}kova, 2008, pp.~60--81. \MR{2509750}

\bibitem[Mur70]{MR261585}
Kunio Murasugi, \emph{On the signature of links}, Topology \textbf{9} (1970),
  283--298. \MR{261585}

\bibitem[MV07]{MR2336253}
Marco Mackaay and Pedro Vaz, \emph{The universal {${\rm sl}_3$}-link homology},
  Algebr. Geom. Topol. \textbf{7} (2007), 1135--1169. \MR{2336253}

\bibitem[Nao06]{MR2263052}
Gad Naot, \emph{The universal {K}hovanov link homology theory}, Algebr. Geom.
  Topol. \textbf{6} (2006), 1863--1892. \MR{2263052}

\bibitem[Par12]{MR2928905}
John Pardon, \emph{The link concordance invariant from {L}ee homology}, Algebr.
  Geom. Topol. \textbf{12} (2012), no.~2, 1081--1098. \MR{2928905}

\bibitem[Ras10]{MR2729272}
Jacob Rasmussen, \emph{Khovanov homology and the slice genus}, Invent. Math.
  \textbf{182} (2010), no.~2, 419--447. \MR{2729272}

\bibitem[Ren24]{MR4843752}
Qiuyu Ren, \emph{Lee filtration structure of torus links}, Geom. Topol.
  \textbf{28} (2024), no.~8, 3935--3960. \MR{4843752}

\bibitem[San25]{sano2025diag}
Taketo Sano, \emph{A diagrammatic approach to the {R}asmussen invariant via
  tangles and cobordisms}, arXiv 2503.05414 (2025).

\bibitem[Sch21]{MR4244204}
Dirk Sch\"{u}tz, \emph{A fast algorithm for calculating {$S$}-invariants},
  Glasg. Math. J. \textbf{63} (2021), no.~2, 378--399. \MR{4244204}

\bibitem[Sch25a]{MR4873797}
Dirk Sch\"utz, \emph{On an integral version of the {R}asmussen invariant},
  Michigan Math. J. \textbf{75} (2025), no.~1, 65--88. \MR{4873797}

\bibitem[{S}ch25b]{schuetz2025var}
Dirk {S}ch{\"u}tz, \emph{On variations of s-invariants from sl(3)-link
  homology}, arXiv 2504.07806 (2025).

\bibitem[Tur20]{MR4079621}
Paul Turner, \emph{Khovanov homology and diagonalizable {F}robenius algebras},
  J. Knot Theory Ramifications \textbf{29} (2020), no.~1, 1950095, 10.
  \MR{4079621}

\end{thebibliography}
\bibliographystyle{amsalpha}

\end{document}